\documentclass[12pt]{article}
\usepackage{amsmath,amsxtra,amssymb,latexsym, amscd,amsthm,pb-diagram,fancyhdr,euscript}
\usepackage{amsthm}
\usepackage{latexsym}
\usepackage{amssymb}
\usepackage{amsmath}
\usepackage{indentfirst}
\usepackage{hyperref}
\usepackage{mathrsfs}
\usepackage{array}
\usepackage{euscript}
\usepackage{slashed}
\usepackage[T1]{fontenc}
\usepackage{graphicx}
\usepackage[utf8]{inputenc}
\usepackage[french,english]{babel}
\usepackage{lipsum} 
\usepackage{multicol}
\hypersetup{
   colorlinks,
    citecolor=blue,
    filecolor=black,
    linkcolor=blue,
    urlcolor=magenta
}
\hypersetup{linktocpage}

\newtheorem{theorem}{Theorem}[section]
\newtheorem{lemma}[theorem]{Lemma}
\newtheorem{proposition}[theorem]{Proposition}

\newtheorem{definition}[theorem]{Definition}

\newtheorem{remark}[theorem]{Remark}

\newcommand{\norm}[1]{\left\Vert#1\right\Vert}

\numberwithin{equation}{section}
\usepackage{amsfonts}

\usepackage{amsmath}
\usepackage{amssymb}
\usepackage{cases}

\def\dive{\mathrm{div}}

\def\curl{\mathrm{Curl}}

\def\geq{\geqslant} 
\def\leq{\leqslant}

\def\C+{C_+([t_0,\infty))}

\begin{document}\mbox{}
\vspace{0.25in}

\begin{center}

{\huge{\bf On attractor's dimensions of the modified Leray-alpha equation}}

\vspace{0.25in}

\large{{\bf PHAM Truong Xuan}\footnote{Faculty of Information Technology, Department of Mathematics, Thuyloi university, Khoa Cong nghe Thong tin, Bo mon Toan, Dai hoc Thuy loi, 175 Tay Son, Dong Da, Ha Noi, Viet Nam. 
\noindent
Email: xuanpt@tlu.edu.vn or phamtruongxuan.k5@gmail.com}}
\&\large{{\bf NGUYEN Thi Van Anh}\footnote{Faculty of Mathematics, Hanoi National University of Education, 136 Xuan Thuy, Cau Giay, Ha Noi, Viet Nam.}}

\end{center}

\begin{abstract} 
The primary objective of this paper is to investigate the modified Leray-alpha equation on the two-dimensional sphere $\mathbb{S}^2$, the square torus $\mathbb{T}^2$ and the three-torus $\mathbb{T}^3$. In the strategy, we prove the existence and the uniqueness of the weak solutions and also the existence of the global attractor for the equation. Then we establish the upper and lower bounds of the Hausdorff and fractal dimensions of the global attractor on both $\mathbb{S}^2$ and $\mathbb{T}^2$. Our method is based on the estimates for the vorticity scalar equations and the stationary solutions around the invariant manifold that are constructed by using the Kolmogorov flows. Finally, we will use the results on $\mathbb{T}^2$ to study the lower bound for attractor's dimensions on the case of $\mathbb{T}^3$.
\end{abstract}

{\bf Keywords.} Modified Leray-alpha equation, $2$-dimensional sphere, square torus, three-torus, global attractor, Hausdorff (fractal) dimension, Kolmogorov flows.

{\bf 2010 Mathematics subject classification.} Primary 35Q30, 76D03, 76F20; Secondary 58A14, 58D17, 58D25, 58D30.

\tableofcontents

\section{Introduction}
The study about the solutions and their asymptotic behaviours of the models of turbulence theory plays an important role to analyse the dynamics of the homogeneous imcompressible fluid flows and many pratical applications. In particular, there is a lot of interest on the three averaged turbulence equations: the Navier-Stokes-alpha, the modified Leray-alpha and the Bardina equations which convergence to the Navier-Stokes equation when the parameter $\alpha$ tends to zero. The existence and uniqueness of weak solutions were established in \cite{CaLuTi,Fo,LaLe,IlLuTi,Mars2001}. The existence of the global attractor and the upper and lower of attractor's Hausdorff and fractal dimensions were studied in \cite{CheHoOlTi,Co,Ily2004,Il2004',IlLuTi}. The existence of the inertial manifold for these equations were obtained recently in\cite{Ti2014,Ko2019,Li}. The algebraic decays in time were given for the Navier-Stokes-alpha equation in \cite{Scho}.

Beside, there are some other works for the equations with damp coeficients such as $2$-D damped-driven Navier-Stokes equations and damped $2$-D and $3$-D Euler-Bardina equations \cite{Ily2004,Il2021,Il21}. In these works, the authors established the well-posedness of the weak solutions and derived the upper and lower bounds of the global attractor's dimensions.

The premilinary method used to study the upper bound of the attractor's dimension is to combine the fundamental theorem about the relation between the Lyapunov exponents and the Hausdorff (fractal) dimension of attractor (see \cite{Il2001,Il2004,Te1988}) and the Leib-Sobolev-Thirring inequality. The lower bound of the dimensions of the global attractor has been studied by using the Kolomogorov flow to construct the family of stationary solution that was given initially for the Navier-Stokes equation in \cite{Liu}. Then this method is developed for the other turbulence equation in \cite{Ily2004,Xuan2021} and the equations with damp coefficients in \cite{Il2021,Il21}.

Concerning the study of Navier-Stokes and averaged turbulence equations on the compact manifolds, there are some works on the attractor's dimensions of the Navier-Stokes and the turbulence equations on the $2$-D closed manifolds such as the sphere $\mathbb{S}^2$ and the square torus $\mathbb{T}^2$. The authors have treated the Navier-Stokes equation in \cite{Il1994,Il1999}, the Navier-Stokes-alpha equation in \cite{Ily2004} and the simplified Bardina equation in \cite{Xuan2021}. 

In the present paper we study the modified Leray-alpha equation on the $2$-D closed manifold ${\bf M}$:
\begin{align}\label{VCM}
\begin{cases}
v_t - \nu \Delta v + v \cdot \nabla u = - \nabla p + f, \cr 
\nabla \cdot v =  0, \cr
v = u - \alpha^2\Delta u,
\end{cases}
\end{align}
where $\nu$ is viscous constant, the velocity $v$ and the filtered function $u$ are unknown which are belong to ${\bf TM}$, $p$ is the pressure and $f$ is external force.

We recall that the $2$-D modified Leray-alpha equation in $\mathbb{R}^2$ with periodic boundary condition was studied in \cite{Ti2014} for the well-posedness and the existence of an inertial manifold. The $3$-D modified Leray-alpha equation in $\mathbb{R}^3$ with periodic boundary condition was studied in \cite{IlLuTi}. The authors established the well-posedness of the weak solution and derivered a upper bound of the dimensions of the global attractor by using the Leib-Sobolev-Thirring in $\mathbb{R}^3$. Recently, the existence of the inertial manifold for the $3$-D equation has established fully in \cite{Ko2019,Li}. 

We will extend and apply the recent work for the simplified Bardina equation of one of the authors \cite{Xuan2021} to consider the modified Leray-alpha equations on $2$-D closed manifold ${\bf M}$ detailized by the sphere $\mathbb{S}^2$ and the square torus $\mathbb{T}^2$. We will establish the well-posedness of the weak solution by the Galerkin approximation method (see Section \ref{S2}). Then, we derive a upper bound of the Hausdorff and fractal dimensions of the global attractor in both the sphere $\mathbb{S}^2$ and tourus $\mathbb{T}^2$ by using the vorticity scalar form of the equation \eqref{VCM} and the generalized theorem of the dimension of attractor on the uniform Lyapunov exponents (see Section \ref{S3}). The lower bound of the dimension in the case of the torus $\mathbb{T}^2$ is obtained by using the Kolomogorov flows to construct the stationary solutions around the invariant manifold (see Section \ref{S3}). In presicely, we will prove in this paper that the upper and lower bounds of the attractor's dimensions are coincided to the ones of the $2$-D simplified Bardina equation and they are improved in comparing with the case of $2$-D Navier-Stokes equation. Finally, we will extend and apply the recent work of Ilyin, Zelik and Kostiano \cite{Il21} to establish the upper bound of the attractor's dimensions in the $3$-D torus (see Section \ref{S5}). The method uses the Squire's transformation to transform the $3$-D equation to the $2$-D case, then apply the results of the lower bound obtained in $2$-D case. Our results with the one obtained in \cite{IlLuTi} complete the two-side estimates of the global attractor's dimensions for the modified Leray-alpha equation in $3$-D case.

This paper is organized as follows: Section \ref{S1} gives some basic formulas and the setting of the modified Leray-alpha equation, Section \ref{S2} discuss the well-posedness of the weak solutions of the equation on the sphere and torus, Section \ref{S} gives the upper and lower bounds of the attractor's dimensions on $\mathbb{T}^2$ and Section \ref{S5} relies on the lower bound on $\mathbb{T}^3$.


\section{Geometrical and analytical setting}\label{S1}
\subsection{Geometric formula and functional spaces}
We recall some geometric formulas on the $2$-dimensional closed manifold $({\bf M},g)$ embedded in $\mathbb{R}^3$ with trivial harnomic forms detailized by the two sphere $\mathbb{S}^2$ and the square torus $\mathbb{T}^2$ (see for details \cite{Il1990,Il1994}). We denote by ${\bf TM}$ the set of tangent vector fields on ${\bf M}$ and by $({\bf TM})^{\bot}$ the set of normal vector fields. We have the definitions of the following operators
$$\curl_n: {\bf TM} \rightarrow ({\bf TM})^{\bot} \, \mbox{and} \, \curl:({\bf TM})^{\bot}\rightarrow {\bf TM}$$
in a neighbourhood of ${\bf M}$ in $\mathbb{R}^3$ as follows:
\begin{definition}
Let $u$ be a smooth vector field on ${\bf M}$ with values in ${\bf TM}$, and let $\vec{\psi}$ be a smooth vector field on ${\bf M}$ with values in $({\bf TM})^\bot$, i.e. $\vec{\psi} = \psi\vec{n}$, where $\vec{n}$ is the outward unit normal vector to ${\bf M}$ and $\psi$ is a smooth scalar function. We then identify the vector field $\vec{\psi}$ with the scalar function $\psi$. Let $\hat u$ and $\hat\psi$ be smooth extensions of $u$ and $\psi$ into a neighbourhood of ${\bf M}$ in $\mathbb{R}^3$ such that $\hat{u}|_{\bf M}=u$ and $\hat{\psi}|_{\bf M}=\psi$. For $x \in {\bf M}$ and $y\in \mathbb{R}^3$, we define
$$\mathrm{Curl}_nu(x) = (\curl\hat{u}(y)\cdot \vec{n}(y))\vec{n}(y)|_{y=x},$$
$$\mathrm{Curl}\vec{\psi}(x) = \curl\psi(x)= \curl\hat{\psi}(y)|_{y=x},$$
where the operator $\curl$ that appears on the right hand sides is the classical $\curl$ operator in $\mathbb{R}^3$.
\end{definition} 
The above definitions of $\curl_nu$ and $\curl\psi$ are independent of the choice of the neighbourhood of ${\bf M}$ in $\mathbb{R}^3$. Moreover, the following formulas hold
\begin{equation}\label{equal1}
\curl_nu = -\vec{n}\dive(\vec{n}\times u), \, \curl\psi= -\vec{n}\times \nabla \psi,
\end{equation}
\begin{equation}\label{equal2.1}
	v \cdot \nabla u + u \cdot \nabla v^T = \nabla (v\cdot u) - v \times \curl_nu,
\end{equation}
\begin{equation}\label{equal3}
\Delta u = \nabla \dive u - \curl\curl_n u,
\end{equation}
where $\times$ is the outer vector product in $\mathbb{R}^3$, $\nabla\psi$ is gradient of the scalar function, $\nabla_vu$ is covariant derivative along the vector field, $\Delta$ is Laplace-de Rham operator defined on the vector fields (see the definition and formula of $\Delta$ in \cite{Il1990}) and $(v\cdot u^T)_i:= \sum_j v_j \partial_i u_j$ in a local basic frame $(x^1,\, x^2)$ of $({\bf M},g)$. 

Let $L^p({\bf M})$ and $L^p({\bf TM})$ be the $L^p$-spaces of the scalar functions and the tangent vector fields on ${\bf M}$ respectively. Let $H^p({\bf M})$ and $H^p({\bf TM})$ be the corresponding Sobolev spaces of scalar functions and vector fields. The inner product on $L^2({\bf M})$ and $L^2({\bf TM})$ are given by
$$\left<u,v\right>_{L^2({\bf M})} = \int_{\bf M}u\bar{v}\mathrm{dVol}_{\bf M}, \, \mbox{for} \, u,v \in L^2({\bf M}),$$
$$\left<u,v\right>_{L^2({\bf TM})} = \int_{\bf M} u \cdot \bar{v} \mathrm{dVol}_{\bf M}, \, \mbox{for} \, u,v \in L^2({\bf TM}).$$
The following integration by parts formulas will be used frequently
$$\left<\nabla h,v\right>_{L^2({\bf TM})}=-\left<h,\dive v\right>_{L^2({\bf M})},$$
$$\left<\curl\vec{\psi},v\right>_{L^2({\bf TM})} = \left<\vec{\psi},\curl_nv\right>_{L^2({\bf M})}.$$

By using Hodge decomposition we have
$$C^\infty({\bf TM}) = \left\{\nabla \psi \, : \, \psi \in C^\infty({\bf M}) \right\} \oplus \left\{ \mathrm{Curl}\psi \, : \, \psi \in C^\infty({\bf M}) \right\}$$
Putting 
$$\mathcal{V} = \left\{ \curl\psi \, : \, \psi \in C^\infty({\bf M}) \right\} \, , \, H = \overline{\mathcal{V}}^{L^2({\bf TM})} \, , \, V = \overline{\mathcal{V}}^{H^1({\bf TM})},$$
with the norms on $H$ and $V$ are
$$\norm{u}^2_H = \left<u,u\right>, \, \norm{u}^2_V = \left<Au,u\right> = \left<\curl_nu, \curl_nu \right>.$$
Since $\dive u = 0$, we have the Poincar\'e inequality
\begin{equation}\label{norm1}
\norm{u}_H \leq \lambda_1^{-1/2} \left( \norm{u}_V + \norm{\dive u}_H \right) = \lambda_1^{-1/2}\norm{u}_V\end{equation}
where $\lambda_1$ is the first eigenvalue of the Stokes operator $A=\curl\curl_n$. We know that
\begin{equation}\label{norm2}
\norm{u}_{H^1({\bf TM})} = \norm{u}^2_{L^2({\bf TM})} + \norm{\dive u}^2_{L^2({\bf M})} + \norm{\curl_n u}^2_{L^2({\bf M})}.
\end{equation}
From the inequalities \eqref{norm1}, \eqref{norm2} and since $\dive u=0$ on $V$, the norms on $H^1$ and $V$ are equivalent for all  $u\in V$. In the rest of this paper, we denote $\norm{.}_{L^2}:=|.|$, $\norm{.}_{V}:=\norm{.}$ and $\norm{.}_{H^1}:=\norm{.}_1$.

\subsection{The modified Leray-alpha equation on $2$-D closed manifolds}
The modified Leray-alpha equation on ${\bf M}$ have the following form
\begin{align}\label{ModCH}
\begin{cases}
v_t - \nu\Delta v + v\cdot \nabla u + \nabla p = f ,\cr
\nabla \cdot u = \nabla \cdot v = 0,\cr
 v = (I-\alpha^2\Delta)u,
\end{cases} 
\end{align}
where $\nu$ is the viscous coefficient, $p$ is the pressure, $f$ is the external force and the unknown functions $u,\, v \in {\bf TM}$.

Using \eqref{equal3} we re-write Equation \eqref{ModCH} as
\begin{align}\label{1ModCH}
\begin{cases}
v_t + \nu\curl\curl_n v + v\cdot \nabla u + \nabla p = f ,\cr
\nabla \cdot u = \nabla \cdot v = 0,\cr
 v = (I-\alpha^2\Delta)u,
\end{cases} 
\end{align}
By using Hodge projection $\mathbb{P}$ on the space $H=\overline{\mathcal{V}}^{L^2({\bf TM})}$ the first equation becomes 
\begin{align}\label{HodEq}
\begin{cases}
v_t + \nu Av + B(v,u)  = f, ,\cr
\nabla \cdot u = \nabla \cdot v = 0,\cr
 v = (I+\alpha^2 A)u,
\end{cases} 
\end{align}
where $B(v,u)=\mathbb{P}(v\cdot \nabla u)$.

On the other hand, if we put $u=-\curl \psi$ and take $\curl_n$ the first equation in \eqref{ModCH} then we obtain the following vorticity scalar form
\begin{equation}\label{ModCH1}
(\Delta\psi_t - \alpha^2\Delta^2\psi_t) - \nu\Delta(\Delta\psi - \alpha^2\Delta^2\psi) + J((I-\alpha^2\Delta)\psi, \Delta \psi) = \curl_n f. 
\end{equation}
Putting $\varphi = \Delta \psi$ we get
\begin{equation}\label{ModCH2}
(\varphi_t - \alpha^2\Delta\varphi_t) - \nu\Delta(\varphi - \alpha^2\Delta\varphi) + J((\Delta^{-1}(I-\alpha^2\Delta)\varphi, \varphi) = \curl_n f. 
\end{equation}
Therefore
\begin{equation}\label{ModCH3}
\varphi_t - \nu\Delta\varphi + (I-\alpha^2\Delta)^{-1} J((\Delta^{-1}(I-\alpha^2\Delta)\varphi, \varphi) = (I-\alpha^2\Delta)^{-1}\curl_n f. 
\end{equation}

The properties of Jacobian operator $J(a,b) = n \times \nabla a \cdot \nabla b$ are given in the following proposition:
\begin{proposition}
On the two-dimensional closed manifold ${\bf M}$ we have
$$J(a,b)=-J(b,a),\,\,\, \int_{\bf M}J(a,b)\, \mathrm{dVol}_{\bf M} = \int_{\bf M}J(a,b)b \,\mathrm{dVol}_{\bf M} = 0$$
and
$$\int_{\bf M}J(a,b)c \,\mathrm{dVol}_{\bf M} = \int_{\bf M}J(b,c)a \,\mathrm{dVol}_{\bf M}.$$
\end{proposition}

\section{Well-posedness and the existence of global attractor}\label{S2}
We consider the existence and uniqueness of the weak solution of the modified Leray-alpha equation under the vectorial form \eqref{HodEq}. The basic method is Galerkin approximation scheme and then passing to the limit using the appropriate Aubin compactness theorems. Since the well-posedness of the $2$-D equation in $\mathbb{R}^2$ with periodic boundary condition was established in \cite{Ti2014} and of the $3$-D equation in $\mathbb{R}^3$ with periodic boundary conditon was treated in \cite{IlLuTi}. Here, we can do by the same way as in \cite{IlLuTi,Ti2014} by establish the $H^1$- and $H^2$-estimates with noting that
$$\left< B(v,u), u \right> = 0.$$
in $H^1$-estimate and the term $\left<B(v,u), Au\right>$ appeared in $H^2$-estimate can be controled by using Young's inequality as
\begin{eqnarray*}
\left|\left<B(v,u), Au\right>\right|_{D(A)'}&\leq& c|v|\norm{v}^{1/2}|A^{3/2}u|\norm{u}\cr
&\leq& c(\lambda_1^{-1}+\alpha^2)|Au|^{1/2}|A^{3/2}u|^{3/2}\norm{u}\cr
&\leq& c(\lambda_1^{-1}+\alpha^2)^4 \frac{\norm{u}^4|Au|^2}{(\nu \alpha^2)^3} + \frac{3\nu\alpha^2}{4}|A^{3/2}u|^2.
\end{eqnarray*}
Therefore, we can get the $H^1$- and $H^2$-estimates as follows (in details see \cite{IlLuTi}):
\begin{equation}\label{H1Es}
|u(t)|^2 + \alpha^2\norm{u(t)}^2 \leq e^{-\nu\lambda_1 t}(|u(0)|^2 + \alpha^2\norm{u(0)}^2) + \frac{K_1}{\nu\lambda_1}(1-e^{-\nu\lambda_1 t}).
\end{equation}
\begin{equation}\label{H2Es}
t(\norm{u(t)}^2 + \alpha^2|Au(t)|^2) \leq \frac{1}{\nu} (tK_1+k_1) + t^2K_2 + (\lambda_1^{-1}+ \alpha^2)^4\frac{2ck_1^2}{(\nu\alpha^2)^4\alpha^4}\left( \frac{t^2K_1}{2} + tk_1 \right).
\end{equation}
Therefore, we can derive the well-posedness of the weak solution of \eqref{HodEq} as in the following theorem.
\begin{theorem}
Let $f \in H$, then for any $T > 0$, Equation \eqref{HodEq} with the initial data $u(0)\in V$ has a unique regular solution $u$ in $[0, T )$. Furthermore, this solution depends continuously on the initial data as a map from $V$ to $C([0, T ], V )$.
\end{theorem}
\begin{proof}
The proof is done by using $H^1$-, $H^2$-estimates and the Galerkin approximation scheme in the same way of \cite[Theorem 3]{IlLuTi}.
\end{proof}
Since the well-posedness, we get a semigroup of solution operators, denoted as $\left\{ S(t) \right\}_{t\geq 0}$, which associates, with each $u_0=u(0)\in V$ , the semi-flow for time $t\geq 0$ : $S(t)u_0 = u(t)$ is unique weak solution of \eqref{HodEq}.

Using the $H^1$-estimate \eqref{H1Es} we can prove the existence of a bounded absorbing ball $B_V(0)$ in $V$.
The compactness of the semigroup $\left\{ S(t) \right\}_{t\geq 0}$ and the existence of bounded absorbing
ball $B_V$ guarantee the existence of the nonempty compact global attractor $\mathcal{A}$.
\begin{theorem}
There is a compact global attractor $\mathcal{A}\subset V$ for Equation \eqref{HodEq}.
\end{theorem}
\begin{proof}
Following Rellich lemma $S_t: V \longrightarrow D(A) \Subset V$, for $t>0$, is a compact semigroup from $V$ into itself. Since $S(t)B_{V}(0) \subset B_{V}(0)$, then the set $C_s:= \overline{\cup_{t\geq s}S(t)B_{V}(0)}^{V}$ is nonempty and compact in $V$. By the monotonic property of $C_s$ for $s>0$ and by the finite intersection property of compact sets, the set
$$\mathcal{A} = \cap_{s>0}C_s \subset V$$
is a nonempty compact set, and also the unique global attractor in $V$.
\end{proof}

\section{Dimensions of global attractor on $2$-D closed manifolds}\label{S}
\subsection{Fundamental theorem on the attractor's dimension}
Let $H$ be a Hilbert space, $X$ be a compact set in $H$ and $S_t$ the nonlinear continuous semigroup generated by the evolution equation
$$\partial_tu = F(u), \, u(0)=u_0,$$
and suppose that
$$S_tX=X \hbox{  for  } t\geq 0.$$
The Hausdorff and fractal dimensions of $X$ are estimated by using the uniform Lyapunov exponents (see \cite{Il2001,Il2004}). 
\begin{definition}
The semigroup $S_t$ is uniformly quasi-differentiable on $X$ for each $t$ if for all $u,\, v \in X$ there exists a linear operator $DS_t(u)$ such that
$$\norm{S_t(u)-S_t(v)-DS_t(u)(u-v)} \leq h(r)\norm{u-v},$$
where $\norm{u-v}\leq r$, $h(r)\rightarrow 0$ as $r\rightarrow 0$ and $\sup_{t\in [0,\, 1]} \sup_{u\in X}\norm{DS_t(u)}_{\mathscr{L}(H,H)}<\infty$.
\end{definition}
The following result is establised in \cite[Theorem 2.1]{Il2001}.
\begin{theorem}\label{TheoremDim}
We assume that the mapping $u \rightarrow S_tu_0$ is uniformly quasi-differentiable in $H$ and its quasi-differentiation is a linear operator $L(t,u_0):\zeta\in H \rightarrow U(t)\in H$, where $U(t)$ is the solution of the first variation equation
\begin{equation}\label{HF1}
\partial_t U = \mathscr{L}(t,u_0)U, \, \, U(0)=\zeta.
\end{equation}
We assume, in addition, that for a fixed $t$ the operator $L(t, u_0) = DS_t(u)$ is compact and norm-continuous with respect
to $u \in X$.

For $N\geq 1, \, n\in \mathbb{N},$ we define $q_N$ by
\begin{equation}\label{Lyapunov}
q_N = \limsup_{t\rightarrow\infty}\sup_{u_0\in X}\sup_{\zeta_i\in H,\norm{\zeta_i}\leq 1, i=1,...,N}\left(  \frac{1}{t}\int_0^t\mathrm{Tr}\mathscr{L}(\tau, u_0)\circ Q_N(\tau)d\tau\right),
\end{equation}
where $Q_N(\tau)$ is the orthogonal projection in $H$ into $\mathrm{Span}\left\{U^1(\tau)...U^N(\tau)\right\}$, and $U^i(t)$ is the solution of \eqref{HF1} with $U^i(0)=\zeta_i$. 

Suppose $q_N \leq f(N)$, where $f$ is concave. The Hausdorff and fractal dimensions of $X$ have the same upper bound
$$\dim_H X \leq \dim_F X\leq N_*,$$
where $N_*\geq 1$ is such that $f(N_*) = 0$.
\end{theorem}
The concave condition of $f$ can be replaced by the condition that the quasi-differential $DS_t(u)$ contracts $N_*$-dimensional volumes uniformly for $u \in X$ (see \cite[Theorem 2.1]{Il2004}).

\subsection{Estimate of the attractor's dimensions}\label{Dim}

\subsubsection{Upper bound}\label{S3}
As the previous sections we denote ${\bf M}$ for both $\mathbb{S}^2$ and $\mathbb{T}^2$. The upper bound of the Hausdorff and fractal dimensions of the global attractor of the $3$-D modified Leray-alpha equation with periodic boundary condition were establised in \cite{IlLuTi} by using the Leib-Sobolev-Thirring inequality. However, we will derive the upper bound of the $2$-D equation on ${\bf M}$ by another method based on the vorticity scalar equation in this section.

We multiply \eqref{ModCH2} by $\varphi$ in $L^2({\bf M})$ we obtain that
\begin{equation*}
\frac{1}{2}\frac{d}{dt}\left( |\varphi|^2 + |\nabla\varphi|^2 \right) + \nu(|\nabla\varphi|^2 + \alpha^2|\Delta\varphi|^2) = \left<\curl_n f,\varphi \right>= \left< f, \curl_n\varphi \right>.
\end{equation*}
Therefore,
\begin{equation*}
\frac{d}{dt} (|\varphi|^2+ \alpha^2|\nabla\varphi|^2) + 2\nu(|\nabla\varphi|^2 + \alpha^2|\Delta\varphi|^2) \leqslant \frac{|f|^2}{\nu} + \nu|\nabla\varphi|^2.
\end{equation*}
Using the Poincar\'e and Gronwall inequalities and integrating with respect to $t$ yield
\begin{equation}\label{Evar}
\limsup_{t\to\infty}|\varphi(t)|^2 \leqslant \frac{|f|^2}{\lambda_1\nu^2}
\end{equation}
and
\begin{equation}\label{Evarphi}
\limsup_{t\to\infty}\frac{1}{t}\int_0^t|\nabla\varphi(\tau)|^2d\tau \leqslant \frac{|f|^2}{\nu^2}.
\end{equation}

We consider the variational equation corresponding to  \eqref{ModCH3}:
\begin{equation}\label{VarEq}
\Phi_t = \Delta \Phi - (I-\alpha^2\Delta)^{-1}J((\Delta^{-1}(I-\alpha^2\Delta)\Phi, \varphi) - (I-\alpha^2\Delta)^{-1}J((\Delta^{-1}(I-\alpha^2\Delta)\varphi, \Phi),
\end{equation}
where $\Phi(0) = \zeta$.

It is standard to show that this equation has a unique solution denoted by
$$L(t, \varphi(0))\zeta := \Phi(t).$$
Using the general theorems in \cite{Te1984} we can show that the semigroup $S_t$ is uniformly quasi-differentiable on the attractor $\mathcal{A}$ of the modified Leray-alpha equation.

Now we establish the Hausdorff and fractal dimensions of the attractor using \eqref{VarEq} in the following theorem:
\begin{theorem}\label{HF}
The Hausdorff and fractal dimension of the attractor $\mathcal{A}$ of the modified Leray-alpha equation on ${\bf M}$ are finite and satisfy 
\begin{equation}\label{Upper1}
\dim_H\mathcal{A} \leqslant \dim_F \mathcal{A} \leqslant G^{2/3} \left(\frac{(4+\epsilon_G)^3}{3L(1+\alpha^2\lambda_1)}(\log G - \frac{1}{2}\log \frac{L}{2})\right)^{1/3},
\end{equation}
\begin{equation}\label{Upper2}
\dim_H\mathcal{A} \leqslant \dim_F \mathcal{A} \leqslant \left( \frac{12}{\sqrt{L(1+\alpha^2\lambda_1)}} \right)^{2/3} G^{2/3} \left( \log G + \frac{1}{2} + \log \frac{3\sqrt{2}}{\sqrt{L (1+\alpha^2\lambda_1)}}\right)^{1/3},
\end{equation}
where $G= \dfrac{|f|}{\nu^2\lambda_1}$ is the Grashof number and $\epsilon_G \to 0$, when $G\to \infty$. In particular, the constant $L=\pi$ in the case of the sphere $\mathbb{S}^2$. 
\end{theorem}
\begin{proof}
Let
$$\mathbb{H} = L^2({\bf M})\cap \left\{ \varphi: \int_{\bf M} \varphi \mathrm{dVol}_{\bf M} =0 \right\} \hbox{  and  } \mathbb{H}^1 = H^1({\bf M}) \cap \mathbb{H}.$$
Putting
$$\left<\left< x, y\right>\right> = \left< x,y\right> - \alpha^2\left< x,\Delta y\right>.$$
In the space $Q_N(\tau)(\mathbb{H})$ we take an orthonormal basis $\left\{\theta_i \right\}_{i=1}^N \subset \mathbb{H}^1$ with norm $\left<\left<.,.\right>\right>$. Now we have
\begin{eqnarray}\label{OperatorTrace1}
&&\mathrm{Tr}\mathscr{L}(\tau,\varphi_0)\circ Q_N(\tau) = \sum_{i=1}^N \left<\left<\mathscr{L}(\tau,\varphi_0)\theta_i, \theta_i \right>\right> \cr
&=& -\nu\sum_{i=1}^N \left<\left< \Delta\theta_i,\theta_i \right>\right> \cr
&&- \sum_{i=1}^N \left<\left<(I-\alpha^2\Delta)^{-1}J(\Delta^{-1}(I-\alpha^2\Delta)\theta_i,\varphi)+ (I-\alpha^2\Delta)^{-1}J(\Delta^{-1}(I-\alpha^2\Delta)\varphi,\theta_i), \theta_i \right> \right>\cr
&=& -\nu\sum_{i=1}^N(|\nabla\theta_i|^2 +|\Delta\theta_i|^2 ) - \sum_{i=1}^N \left<J(\Delta^{-1}(I-\alpha^2\Delta)\theta_i,\varphi)+ J(\Delta^{-1}(I-\alpha^2\Delta)\varphi,\theta_i), \theta_i \right> \cr
&=& -\nu\sum_{i=1}^N (|\nabla\theta_i|^2 +|\Delta\theta_i|^2 )   - \sum_{i=1}^N \left<J(\Delta^{-1}\theta - \alpha^2\theta_i,\varphi), \theta_i \right>\cr
&\leq& -\nu\sum_{i=1}^N (|\nabla\theta_i|^2 +|\Delta\theta_i|^2 )   - \int_M \sum_{i=1}^N \theta_i (n\times\nabla(I - \alpha^2\Delta)^{-1}\theta_i)\cdot \nabla\varphi dx \cr
&&+  \alpha^2\sum_{i=1}^n\left< J(\theta_i,\varphi),\theta_i\right>\cr
&\leq& -\nu\sum_{i=1}^N (|\nabla\theta_i|^2 +|\Delta\theta_i|^2 )   + \int_M \left(\sum_{i=1}^N \theta^2_i \right)^{1/2} \left(\sum_{i=1}^N |v_i|^2 \right)^{1/2}|\nabla\varphi| dx\cr
&&+ \alpha^2\sum_{i=1}^n\left< J(\varphi,\theta_i),\theta_i\right>\cr
&\leq& -\nu\sum_{i=1}^N (|\nabla\theta_i|^2 +|\Delta\theta_i|^2 )   + \norm{\rho}^{1/2}_\infty\left(\sum_{i=1}^N |\theta_i|^2 \right)^{1/2}|\nabla\varphi| \,\,\hbox{   (due to   } \int_{\bf M} J(\varphi,\theta_i)\theta_i \mathrm{dVol_{\bf M}} = 0) \cr
&\leq& -\nu\sum_{i=1}^N (|\nabla\theta_i|^2 +|\Delta\theta_i|^2 )  + \norm{\rho}^{1/2}_\infty N^{1/2}|\nabla\varphi|,
\end{eqnarray}
where
$$\rho(s) = \sum_{i=1}^N|v_i(s)|^2 = \sum_{i=1}^n |n \times \nabla(\Delta - \alpha^2\Delta^2)^{-1}\theta_i|^2.$$
The following estimate of the function $\rho$ on the $2$-D closed manifold ${\bf M}$ is valid (for details see \cite[Appendix]{Xuan2021}).
\begin{eqnarray}\label{INE}
&&2\sqrt{L(1+\alpha^2\lambda_1)}\norm{\rho}_\infty^{1/2} \leqslant (2\log(k+1)+1)^{1/2} + \sqrt{2}(k+1)^{-1}\left( \lambda_1^{-1}\sum_{i=1}^N|\nabla\theta_i|^2 \right)^{1/2} \cr
&\leqslant& (2\log(k+1)+1)^{1/2} + \sqrt{2}(k+1)^{-1}\left( \lambda_1^{-1}\sum_{i=1}^N(|\nabla\theta_i|^2 + \alpha^2|\Delta\theta_i|^2) \right)^{1/2},
\end{eqnarray}
where $k$ is a positive integer and $L$ is a positive constant ($L=\pi$ in the case of $\mathbb{S}^2$).

Since on the $S^2$ the eigenvalues of $\Delta$ are $\lambda_n = n(n+1)$ of multiplicity $2n+1$ for $n=1,2,...$, we have
$$T(t,\varphi_0):= \sum_{i=1}^N(|\nabla\theta_i|^2 + \alpha^2|\Delta\theta_i|^2) \geqslant \sum_{i=1}^N  \lambda_i \geqslant \frac{\lambda_1}{4} N^2.$$
Hence
$$N \leqslant 2((\lambda_1)^{-1}T)^{1/2}.$$
Equation \eqref{OperatorTrace1} implies now,
\begin{eqnarray*}
&&\mathrm{Tr}\mathscr{L}(\tau,\varphi_0)\circ Q_N(\tau) \leqslant -\nu\lambda_1 (\lambda_1^{-1} T) \cr
&&+ L^{-1/2}(1+\alpha^2\lambda_1)^{-1/2} \left( (2\log(k+1)+1)^{1/2} + \sqrt{2}(k+1)^{-1}(\lambda_1^{-1}T) \right)(\lambda_1^{-1}T)^{1/4}|\nabla\varphi|^2.
\end{eqnarray*}
Since we obtain the same bounded function of $\mathrm{Tr}\mathscr{L}(\tau,\varphi_0)\circ Q_N(\tau)$ such as the one of the simplified Bardina equation, the rest of the proof can be done by the same way in \cite[Theorem 4.4]{Xuan2021} and we get the upper bounds \eqref{Upper1} and \eqref{Upper2} in our theorem.
\end{proof}
\begin{remark}\label{Remark}
In the above theorem we prove that the upper bound of the Hausdorff and fractal dimensions of the global attractor is coincided to the ones of the simplified Bardina equation obtained in \cite{Xuan2021}. In particular, as $\alpha$ tends to zero we get the same upper bound of the Haussdorff and fractal dimensions of the global attractor for the Navier-Stokes equation on $\mathbb{S}^2$ (see \cite{Il1994,Il1999}). Our theorem can be also extended to the two dimensional closed manifolds which have the non trivial harmonic forms as well as \cite[Theorem 4.6]{Xuan2021}.
\end{remark}

\subsubsection{Lower bound}\label{S4}
Since a global attractor is a maximal strictly invariant compact set, it follows that the attractor contains the unstable manifolds of stationary points, that is the invariant manifolds along which the solutions convergence exponentially to the stationary points as $t$ tends to infinity. From this point we can establish the lower bound of the attractor's dimension on the square torus $\mathbb{T}^2=[0; \, 2\pi]\times[0; \, 2\pi]$ by constructing a family of stationary solutions arising from the family of Kolmogorov flows. Recall that the scalar vorticity form of the equation is
\begin{equation*}
(\varphi_t - \alpha^2 \Delta \varphi_t) - \nu\Delta(\varphi - \alpha^2 \Delta \varphi) + J(\Delta^{-1}(I-\alpha^2\Delta)\varphi,\varphi) = \curl_n f. 
\end{equation*}
Putting $\psi = \varphi -  \alpha^2\Delta\varphi$, then
\begin{equation}\label{LowerEq}
\psi_t - \nu\Delta\psi + J(\Delta^{-1}\psi,(I-\alpha^2\Delta)^{-1}\psi) = \curl_n f.
\end{equation}

We consider the following family of forces depending on the integer parameter $s$:
\begin{align*}
f=f_s =\begin{cases}
f_1 = \frac{1}{\sqrt{2}\pi}\nu^2\lambda s^2\sin s x_2,\\
f_2 =0,
\end{cases}
\end{align*}
where we choose the parameter $\lambda:=\lambda(s)$ later. Then, we have
$$|f| = \nu^2\lambda s^2, \, G= \lambda s^2$$
and
\begin{equation}\label{flow}
\curl_n f_s = F_s = -\frac{1}{\sqrt{2}\pi}\nu^2\lambda s^3\cos sx_2, \, |\curl_nf| = \nu^2\lambda s^3.
\end{equation}
Corresponding to the family \eqref{flow} is the family of stationary solutions
\begin{equation*}
\psi_s = -\frac{1}{\sqrt{2}\pi}\nu\lambda s\cos sx_2
\end{equation*}
of Equation \eqref{LowerEq} due to $\psi_s$ depends only on $x_2$, the nonlinear term vanishes
$$J(\Delta^{-1}\psi_s,(I-\alpha^2\Delta)^{-1}\psi_s)=0$$
and the equality $-\nu \Delta\psi_s = F_s$ is verified directly.

We linearize \eqref{LowerEq} about the stationary solution \eqref{flow} and consider the eigenvalue
problem
\begin{eqnarray}\label{EP}
\mathcal{L}_s\psi :&=& J(\Delta^{-1}\psi_s,(I-\alpha^2\Delta)^{-1}\psi) \cr
&&+ J(\Delta^{-1}\psi,(I-\alpha^2\Delta)^{-1}\psi_s) - \nu\Delta\psi = -\sigma\psi.
\end{eqnarray}
We use the orthonormal basis of trigonometric functions, which are the eigenfunctions of the Laplacian on the two-dimensional torus,
$$\left\{ \frac{1}{\sqrt{2}\pi}\sin kx, \frac{1}{\sqrt{2}\pi}\cos kx  \right\}, \, kx = k_1x_1+k_2x_2,$$
$$k \in \mathbb{Z}^2_+ = \left\{ k\in \mathbb{Z}_0^2| k_1\geq 0,\, k_2\geq 0 \right\}\cup \left\{k\in \mathbb{Z}_0^2| k_1\geq 1, k_2\leq 0  \right\}$$
and we rewrite $\psi$ as a Fourier series
$$\psi = \frac{1}{\sqrt{2}\pi}\sum_{k\in \mathbb{Z}_+^2} a_k \cos kx + b_k \sin kx.$$
Since $J(a,b)=-J(b,a)$, we have
\begin{eqnarray*}
&&J(\Delta^{-1}\cos sx_2, (I-\alpha^2\Delta)^{-1}\cos kx) + J(\Delta^{-1}\cos kx, (I-\alpha^2\Delta)^{-1}\cos sx_2)\cr
&=& \frac{\nu \lambda s}{\sqrt{2}\pi}\left( \frac{1}{s^2}\frac{1}{1+\alpha^2k^2} - \frac{1}{k^2}\frac{1}{1+\alpha^2 s^2} \right) J(\cos sx_2, a_k cos kx + b_k \sin kx)\cr
&=& \frac{\nu \lambda s}{\sqrt{2}\pi}\frac{k^2-s^2}{(s^2+\alpha^2s^4)(k^2+\alpha^2k^4)}J(\cos sx_2, a_k cos kx + b_k \sin kx).
\end{eqnarray*}
Plugging this into \eqref{EP} we obtain that
\begin{eqnarray}\label{EP1}
\frac{\lambda s}{\sqrt{2}\pi (s^2+ \alpha^2s^4)}&&\sum_{k\in \mathbb{Z}_+^2} \left( \frac{k^2-s^2}{k^2+\alpha^2k^4} \right)J(\cos sx_2, a_k \cos kx + b_k \sin kx)+\cr
&&+ \sum_{k\in \mathbb{Z}_+^2}(k^2+ \hat{\sigma})(a_k\cos kx + b_k \sin kx)=0,
\end{eqnarray}
where $\hat{\sigma}=\sigma/\nu$.

We can calculate that
\begin{eqnarray*}
J(\cos s x_2, \cos(k_1x_1+k_2x_2)) &=& -k_1s \sin sx_2 \sin(k_1x_1 + k_2x_2)\cr
&=& \frac{k_1 s}{2} (\cos (k_1x_1 +(k_2+s)x_2)) -\cos (k_1x_1+(k_2-s)x_2)
\end{eqnarray*}
and
\begin{eqnarray*}
J(\cos s x_2, \sin(k_1x_1+k_2x_2)) &=& k_1s \sin sx_2 \cos(k_1x_1 + k_2x_2)\cr
&=& \frac{k_1 s}{2} (\sin (k_1x_1 +(k_2+s)x_2)) - \sin (k_1x_1+(k_2-s)x_2).
\end{eqnarray*}
Substituting these equalities into \eqref{EP1} and regroup the terms with $\cos(k_1x_1+k_2x_2)$, we get the following equation for the coefficients $a_{k_1,k_2}$ 
\begin{eqnarray*}
&&-\Lambda(s)k_1 \left( \frac{k_1^2+(k_2+s)^2-s^2}{k_1^2+(k_2+s)^2 + \alpha^2(k_1^2 + (k_2+s)^2)^2} \right) a_{k_1k_2+s}\cr
&&+\Lambda(s)k_1 \left( \frac{k_1^2+(k_2-s)^2-s^2}{k_1^2+(k_2-s)^2 + \alpha^2(k_1^2 + (k_2-s)^2)^2} \right) a_{k_1k_2-s} + (k^2+\hat{\sigma})a_{k_1k_2} =0,
\end{eqnarray*}
where 
\begin{equation}
\Lambda = \Lambda(s):= \frac{s^2\lambda}{2\sqrt{2}\pi (s^2+ \alpha^2s^4)} = \frac{\lambda }{2\sqrt{2}\pi(1+\alpha^2 s^2)}.
\end{equation}
Similarly the equation for $b_{k_1,k_2}$ has also this form.

We put
$$a_{k_1k_2} \left( \frac{k^2-s^2}{k^2+ \alpha^2k^4} \right) =: c_{k_1k_2}.$$
and
$$k_1 = t, \, k_2= sn+r, \hbox{  and  } c_{t \, sn+r}= e_n,$$
$$t=1,2,..., \, r \in \mathbb{Z}, \, r_{\min} < r < r_{\max},$$
where the numbers $r_{\min}$ and $r_{\max}$ satisfy that $r_{\max} - r_{\min} <s$ and will be specified below we obtain for each $t$ and $r$ the following three term recurrence relation:
\begin{equation}\label{d1}
d_ne_n + e_{n-1} - e_{n+1} =0 , \, n=0,\pm 1,\pm 2,...,
\end{equation}
where
\begin{equation}\label{d2}
d_n = \frac{(t^2+(sn+r)^2+ \alpha^2(t^2+(sn+r)^2)^2)(t^2+(sn+r)^2+\hat{\sigma})}{\Lambda t(t^2+(sn+r)^2-s^2)}.
\end{equation}
We look for non-trivial decaying solutions $\left\{ e_n \right\}$ of \eqref{d1} and \eqref{d2}. Each nontrivial decaying solution with $\mathrm{Re}(\hat{\sigma})>0$ produces an unstable eigenfunction $\psi$ of the eigenvalue problem \eqref{EP}.  
\begin{theorem}\label{THLower2D}
Given an integer $s>0$ let a pair of integers $t,\, r$ belong to a bounded region $A(\delta)$ given by
\begin{gather}\label{Con}
t^2+r^2<s^2/3, \, t^2+(-s+r)^2>s^2, \, t^2+(s+r)^2>s^2, \, t\geqslant \delta s,\cr
r_{\min}<r<r_{\max}, \, r_{\min} = -s/6, \, r_{\max}=s/6, \, 0<\delta<1/\sqrt{3}.
\end{gather}
For any $\Lambda=\frac{\lambda }{2\sqrt{2}\pi(1+\alpha^2 s^2)}>0$ there exists a unique real eigenvalue $\hat{\sigma} = \hat{\sigma}(\Lambda)$, which increases monotonically as $\Lambda\to \infty$ and satisfies the following inequality
\begin{equation}\label{Estimate}
c_1(\alpha,t,r,s)\Lambda < \hat{\sigma} < c_2(\alpha,t,r,s)\Lambda. 
\end{equation}
The unique $\Lambda_0 = \Lambda_0(s)$ solving the equation
$$\hat{\sigma}(\Lambda_0) = 0$$
satisfes the two-sided estimates
\begin{gather}\label{LU}
\frac{1}{\sqrt 2}\delta^2s(1+\alpha^2s^2) < \Lambda < \frac{55\sqrt 5}{63\sqrt 2}\frac{s(1+\alpha^2s^2)}{\delta^2} \hbox{    for    } \alpha \geqslant 0,\cr
\frac{1}{\sqrt 2}\delta^2s < \Lambda < \frac{5}{3\sqrt 3}\frac{s}{\delta^2} \hbox{    for    } \alpha =0.
\end{gather}
In the term of $\lambda$ these inequalities are
\begin{gather*}
2\pi\delta^2s(1+\alpha^2s^2)^2 < \lambda < \frac{110\sqrt 5\pi}{63}\frac{s(1+\alpha^2s^2)^2}{\delta^2} \hbox{    for    } \alpha \geqslant 0,\cr
2\pi\delta^2s < \lambda < \frac{20\pi}{3\sqrt 6}\frac{s}{\delta^2} \hbox{    for    } \alpha =0.
\end{gather*}
\end{theorem}
\begin{proof}
The proof is done similarly \cite[Theorem 4.8]{Xuan2021} and we obmit.
\end{proof}

In the rest we give the lower bound of the attractor's dimension by using the above theorem.
Since
$$\Lambda=\frac{\lambda }{2\sqrt{2}\pi(1+\alpha^2 s^2)},$$
we rewrite \eqref{LU} in the term of $\lambda(s)$ to see that for
$$\lambda_{\alpha\geq 0} = \frac{110\sqrt{5}\pi}{63}s\delta^{-2}(1+\alpha^2s^2)^2,$$
$$\lambda_{\alpha=0} = \frac{20\pi}{3\sqrt{6}}s\delta^{-2},$$
each point in $(t,\, r)$-plane satisfying \eqref{Con} produces an unstable (positive) eigenvalue
$\hat{\sigma}>0$ of multiplicity two (the equation for the coefficients $b_k$ is the same).
Denoting by $d(s)$ the number of points of the integer lattice inside the region $A(\delta)$ we obviously have
\begin{equation}
d(s):= \sharp \left\{ (t,r)\in D(s) = \mathbb{Z}^2\cap A(\delta) \right\} \simeq a(\delta)s^2 \hbox{   as   } s \to \infty,
\end{equation}
where $a(\delta)s^2=|A(\delta)|$ is the area of the region $A(\delta)$. Therefore
the dimension of the unstable manifold around the stationary solution $\psi_s$ is
at least $2a(\delta)s^2$ and we obtain that
\begin{equation}\label{ED}
\dim \mathcal{A} \geqslant 2d(s)\simeq 2a(\delta) s^2.
\end{equation}

It is reasonable to consider two case:\\
{\bf The case $\alpha=0$.}

We have 
$$G = \lambda_{\alpha=0} s^2 =  \frac{20\pi}{3\sqrt{6}}s^3\delta^{-2}$$
and writing the estimate \eqref{ED} in terms of the Grashof number $G$ we obtain
\begin{eqnarray*}
\mathrm{dim}\mathcal{A} &\geqslant& 2a(\delta) s^2 \simeq 2\left( \frac{3\sqrt{6}}{20\pi} \right)^{2/3} a(\delta)\delta^{4/3}G^{2/3}\cr
\mathrm{dim}\mathcal{A} &\geqslant& 2\left( \frac{3\sqrt{6}}{20\pi} \right)^{2/3} (\max_{0<\delta<1/\sqrt{3}} a(\delta)\delta^{4/3}) G^{2/3} = 0,006 G^{2/3},
\end{eqnarray*}
where $\max_{0<\delta<1/\sqrt{3}} a(\delta)\delta^{4/3} = 0,012$. This is exact the same lower bound obtained for the global attractor's dimensions of the Navier-Stokes equation (see \cite{Liu,Il2004'}).\\
{\bf The case $0< \alpha \ll 1$.}

Here we can obtain the following lower bound for $G \thicksim (1/\alpha)^3$. Let $0<s<1/\alpha$. Then $1+\alpha^2s^2<2$ and
$$G \leqslant \frac{440\sqrt{5}\pi}{63}s^3\delta^{-2}$$
and by the same way as above we obtain that
$$\mathrm{dim}\mathcal{A} \geqslant 2\left( \frac{63}{440\sqrt{5}\pi} \right)^{2/3}(\max_{0<\delta<1/\sqrt{3}} a(\delta)\delta^{4/3}) G^{2/3} = 0,0018 G^{2/3}.$$
In particular, setting $s\simeq 1/\alpha$ we can obtain in term of $\gamma$ that
$$C_1\frac{1}{\alpha^2} \leq \mathrm{dim}\mathcal{A} \leqslant C_2\frac{1}{\alpha^2}\left( \log \frac{1}{\alpha} \right)^{1/3}.$$

\section{The lower bound of global attractor on $\mathbb{T}^3$}\label{S5}
In this section we will develop the method of Ilyin, Zelik and Kostiano in a recent work \cite{Il21} to give the lower bound of the global attractor for the modified Leray-alpha equation on $\mathbb{T}^3=[0,\, 2\pi]^3$. The method uses the Squire's transformation to transform the $3$-D instability analysis to the instability analysis of the transformed $2$-D problem which has obtained in the previous section. To avoid the confusion we denote the unknowns by $\vec{u}$, the components by $u$ and the covariant derivative by $\nabla_x$. 

\subsection{The stationary solutions}
Now we consider the modified Leray-alpha equation \eqref{ModCH} on $\mathbb{T}^3$ with the right hand sides are given by
\begin{align}\label{flow1}
f=f_s =\begin{cases}
f_1 = \frac{1}{\sqrt{2}\pi}\nu^2\lambda s^2\sin s x_3,\cr
f_2 =0,\cr
f_3=0,
\end{cases}
\end{align}
where $\lambda = \lambda(s)$ is chosen latter. Then, we have
$$|f| = \nu^2\lambda s^2, \, G= \lambda s^2$$
and
\begin{equation*}
\curl_n f_s = F_s = -\frac{1}{\sqrt{2}\pi}\nu^2\lambda s^3\cos sx_3, \, |\curl_nf| = \nu^2\lambda s^3.
\end{equation*}
The family of stationary solutions of\eqref{ModCH} corresponding to \eqref{flow1} are
\begin{align}\label{stasol}
\vec{v}_0(x_3) =\begin{cases}
v_0(x_3) = \frac{1}{\sqrt{2}\pi}\nu\lambda \sin s x_3,\cr
0,\cr
0
\end{cases}
\end{align}
Moreover, $\vec{u} = (I-\alpha^2\Delta_x)^{-1}\vec{v}_0 = (u_0,0,0)^T$ depends only on $x_3$ hence $\vec{v}_0\cdot \nabla_x\vec{u}_0 =0$.

We derive the linearized equation of \eqref{ModCH} on the stationary solutions \eqref{stasol} as follows
\begin{align}\label{Lineareq}
\begin{cases}
\partial_t\omega + u_0 \frac{\partial\bar{\omega}}{\partial x_1} + \bar{\omega}_3 \frac{\partial u_0}{\partial x_3}e_1 -\Delta_x\omega + \nabla_x q = 0,\cr
\dive \omega = 0,
\end{cases}
\end{align}
where $e_1=(1,0,0)^T$ and $\bar{\omega} = (I-\alpha^2\Delta_x)^{-1}\omega$ with the assumption
$$\int_{\mathbb{T}^3}\omega(x,t)\mathrm{dx} =0.$$
We consider the solution of \eqref{Lineareq} in the following form
\begin{equation}\label{SolLinear}
\omega(x,t) = (\omega_1(x_3),\omega_2(x_3),\omega_3(x_3))^Te^{i(ax_1+bx_2-act)} \hbox{   and   } q(t)=q(x_3)e^{i(ax_1+bx_2-act)},
\end{equation}
where $a,b \in \mathbb{Z}$ satisfied that $\omega$ and $q$ are $2\pi$-periodic in each $x_i$.

If there exist a solution \eqref{SolLinear} of Equation \eqref{Lineareq}, then at $t=0$ we have that
$$\omega(x,0) = (\omega_1(x_3),\omega_2(x_3),\omega_3(x_3))^T e^{i(ax_1+bx_2)}$$
is a vector-valued eigenfunction of the stationary operator
\begin{equation}\label{SO}
L_3(\vec{v}_0)\omega = u_0\frac{\partial\bar{\omega}}{\partial x_1} + \bar{\omega}_3\frac{\partial u_0}{\partial x_3}e_1 -\Delta_x \omega + \nabla_x q
\end{equation}
and $iac$ is the corresponding eigenvalue. If $\Re(iac) < 0$, then the corresponding mode is unstable.

Plugging \eqref{SolLinear} into \eqref{Lineareq} we obtaint that
\begin{align}\label{Lineareq'}
\begin{cases}
\Delta_x \omega_1 - ia(u_0\bar{\omega}_1-c\omega_1) = iaq + \bar{\omega}_3 u'_0,\cr
\Delta_x \omega_2 - ia(u_0\bar{\omega}_2-c\omega_2) = ibq,\cr
\Delta_x \omega_3 - ia(u_0\bar{\omega}_3-c\omega_3) = q',\cr
ia\omega_1 + ib\omega_2 + \omega'_3=0,
\end{cases}
\end{align}
where we denote $':= \partial / \partial x_3$.
\begin{lemma}\label{unsSol}
There are no unstable solutions of equation \eqref{Lineareq} which can be written by \eqref{SolLinear} at $a=0$.
\end{lemma}
\begin{proof}
The proof is a slightly modification of \cite[Lemma 5.1]{Il21} for replacing $-\gamma$ by $\Delta$. 
Let $a=0$ we have that
$$\omega(x,t) = (\omega_1(x_3),\omega_2(x_3),\omega_3(x_3))^Te^{ibx_2} \hbox{   and   } q(t)=q(x_3)e^{ibx_2}$$
is solution of \eqref{Lineareq}. Moreover, Equation \eqref{Lineareq'} becomes
\begin{align*}
\begin{cases}
\Delta_x \omega_1 + iac\omega_1 = iaq + \bar{\omega}_3 u'_0,\cr
\Delta_x \omega_2 + iac\omega_2 = ibq,\cr
\Delta_x \omega_3 + iac\omega_3 = q',\cr
ib\omega_2 + \omega'_3=0,
\end{cases}
\end{align*}
The final equation leads to $\omega_2 = -\dfrac{\omega'_3}{ib}$. Plugging this into the second equation we get
$$\omega_3''' + iac\omega_3' = b^2q.$$
Differentiating the third with respect to $x_3$ we obtain 
$$\omega_3''' + iac\omega_3' = q''.$$
Therefore, we have that $q''=b^2q$, hence $q=0$ due to $q$ is periodic.

Since we considering for unstable solutions, it follows that $\Re(ic) < 0$. This leads to $\mathrm{Ker}_{L^2}(\Delta + ic) = \left\{0\right\}$. This gives that $\omega_2 = \omega_3 = 0$, and, finally, $\omega_1 = 0$.

If $a = b = 0$, then $\omega_3'=0$, then $\omega_2 = 0$ by periodicity and zero mean condition.
This shows that $q = 0$ and $\omega_1 = \omega_2 = 0$. Our proof is completed.
\end{proof}

\subsection{Transform from $\mathbb{T}^3$ to $\mathbb{T}^2$}
Now we use the Squire's transformation to transform the eigenfunctions of $L_3(\vec{v}_0)$ on $\mathbb{T}^3$ to the ones of $L_2(\vec{v}_0)$ on the $2$-D torus. The idea and detailized techniques are given in \cite{Il21}.

Since Lemma \eqref{unsSol}, we assume that $a\ne 0$ in \eqref{Lineareq'}. Multiplying the first equation in \eqref{Lineareq'} by $a$ and the second by $b$ a adding up the obtained results we get
\begin{align}\label{Lineareq''}
\begin{cases}
\widehat{\Delta}_x \omega_1 - i\widehat{a}(u_0\bar{\widehat{\omega}}_1- \widehat{c}\widehat{\omega}_1) = i\widehat{a}\widehat{q} + \bar{\widehat{\omega}}_3 u'_0,\cr
\widehat{\Delta}_x \omega_3 - i\widehat{a}(u_0\bar{\widehat{\omega}}_3- \widehat{c}\widehat{\omega}_3) = \widehat{q}',\cr
i\widehat{a}\widehat{\omega}_1 + \widehat{\omega}'_3=0,
\end{cases}
\end{align}
where
\begin{eqnarray}\label{relations}
&&\widehat{a}^2 = a^2+b^2, \, \widehat{\omega}_1 = \frac{a\omega_1+b\omega_2}{\widehat{a}},\, \widehat{\omega}_3 = \omega_3,\cr
&&\widehat{\Delta} = \frac{\widehat{a}}{a}\Delta, \, \widehat{q}=q\frac{\widehat{a}}{a},\, \widehat{c}=c.
\end{eqnarray}
The solutions of the problem \eqref{Lineareq''} on the $2$-D torus 
$$\mathbb{\widehat{T}}_a^2 = \left\{ (x_1,x_3) \in [0,2\pi/|\widehat{a}|]\times [0,2\pi]\right\}$$
have the following form
\begin{equation}\label{Sol2D}
\widehat{\omega}(x_1,x_3,t) = (\widehat{\omega}_1(x_3),\widehat{\omega}_3(x_3))^T e^{i(\widehat{a}x_1-\widehat{a}\widehat{c}t)},\, \widehat{q}(x_1,x_3,t) = q(x_3)e^{i(\widehat{a}x_1-\widehat{a}\widehat{c}t)}.
\end{equation}
Observe that if Equation \eqref{Lineareq''} has the solutions \eqref{Sol2D}, then the vector function
\begin{equation}\label{eigenfunction}
\widehat{\omega}(x_1,x_3,0) = (\widehat{\omega}_1(x_3),\widehat{\omega}_3(x_3))^T e^{i\widehat{a}x_1}
\end{equation}
is a vector-valued eigenfunction with eigenvalue $i\widehat{a}\widehat{c}$ of the stationary operator
\begin{equation}\label{twoDoperator}
L_2(\vec{v}_0)\widehat{\omega} = -\widehat{\Delta}\widehat{\omega} + u_0\frac{\partial\bar{\widehat{\omega}}}{\partial x_1} + \bar{\widehat{\omega}}_3\frac{\partial u_0}{\partial x_3}e_1 + \nabla_x\widehat{q},\, \dive{\widehat{\omega}}=0
\end{equation}
on $\widehat{\mathbb{T}}^2_a$, where $u_0 = (I-\alpha^2\Delta_x)^{-1}v_0$. The stationary solution and the generating right-hand side are
\begin{align}\label{SOL}
\vec{v}_0(x_3)=
\begin{cases}
v_0(x_3) = \frac{1}{\sqrt{2}\pi}\nu\lambda \sin s x_3\cr
0
\end{cases}
\end{align}
and
\begin{align}\label{Force}
\widetilde{f}_s(x_3)=\widehat{\Delta}_x v_0(x_3)=
\begin{cases}
f_1(x_3) = \frac{\widehat{a}}{a\sqrt{2}\pi}\nu^2\lambda s^2\sin s x_3\cr
0
\end{cases}
\end{align}
We suppose that $\widehat{a}>0$. The result on the Squire's reduction of the $3$-D instability analysis to the $2$-D case is given in the following lemma.
\begin{lemma}\label{Lem}
Let $\widehat{\omega}$ in \eqref{eigenfunction} be an unstable eigenfunction of the operator \eqref{twoDoperator} on the $2$-D torus $\widehat{\mathbb{T}}^2_a = [0,2\pi/\widehat{a}]\times[0,2\pi]$. Then for any pair of integers $a, b \in \mathbb{Z}$ with
$$a^2+b^2=\widehat{a}^2$$
there exist an unstable solution of system \eqref{Lineareq'} on three-torus $\mathbb{T}^3=[0,2\pi]^3$.
\end{lemma}
\begin{proof}
By using the relations \eqref{relations} we can find $q,\, \omega_3, \, c$.
Observe that the second equation in \eqref{Lineareq'} is
\begin{equation*}
(\Delta_x+iac)\omega_2 - iau_0(I-\alpha^2\Delta)^{-1}\omega_2 =ibq.
\end{equation*}
This is equivalent to
\begin{equation}\label{2}
-[\Delta_x+iac - iau_0(I-\alpha^2\Delta)^{-1}]\omega_2 = -ibq.
\end{equation}
Considering the following sesquilinear form $\mathbb{A}$ on $H^1_0([0,2\pi],{\bf M})\times H^1_0([0,2\pi],{\bf M})$: 
$$\mathbb{A}(x,y) = -[\Delta_x+iac - iau_0(I-\alpha^2\Delta)^{-1}x,y].$$
Clearly, $\left|\mathbb{A}(x,y)\right|$ is bounded by $\norm{x}_{H^1_0}\norm{y}_{H^1_0}$. Moreover, we have that
$$\left|\mathbb{A}(\omega_2,\omega_2)\right| \geq \norm{\nabla_x\omega_2}_{L^2}^2 - \Re{(iac)}\norm{\omega_2}^2_{L^2}.$$
Since $\widehat{\omega}$ is unstable, we have $\Re{(iac)}<0$. Therefore, the linear operator $\mathbb{A}$ is coercive.
By using Lax-Milgram theorem (in complex) (see \cite[Theorem 7]{Ba}), there exists a bounded and inverted operator $\widetilde{\mathbb{A}}: H_0^1([0,2\pi],{\bf M})\to H^{-1}([0,2\pi],{\bf M}) $ such that
$$\mathbb{A}(x,y) = \left< \widetilde{\mathbb{A}}x,y\right>.$$
Therefore, Equation \eqref{2} becomes $\widetilde{\mathbb{A}}\omega_2 = -ibq$ and it
has a unique solution $\omega_2= \widetilde{\mathbb{A}}^{-1}(-ibq)\in H^1_0([0,2\pi],{\bf M})$. Finally, we obtain that 
$$\omega_1 = \frac{\widehat{a}\widehat{\omega}_1 - b\omega_2}{a}.$$ 
\end{proof}

\subsection{Lower bound on $\mathbb{T}^3$}
In this section we apply the lower bound of the global attractor obtained on $2$-D torus $\mathbb{T}^2$ to establish the one on $\mathbb{T}^3$. We denote the second coordinate by $x_3$, so that $x_1, x_3$ are the coordinates on $\mathbb{T}^2$. The linearized stationary operator is \eqref{twoDoperator} with the family of the forcing terms are \eqref{Force}, and the corresponding stationary solutions are \eqref{SOL}.

Applying $\curl$ to \eqref{twoDoperator} we obtain the equivalent scalar operator in terms of the vorticity in the previous Section \ref{S4} on $\mathbb{T}^2$: 
\begin{eqnarray}\label{Lagrange}
\mathcal{L}_s\omega :&=& J(\Delta^{-1}\omega_s,(I-\alpha^2\Delta)^{-1}\omega) \cr
&&+ J(\Delta^{-1}\omega,(I-\alpha^2\Delta)^{-1}\omega_s) - \nu\Delta\omega = -\sigma\omega.
\end{eqnarray}
where
\begin{equation*}
\omega_s = \curl_n\vec{v}_0 = -\frac{1}{\sqrt{2}\pi}\nu\lambda s\cos sx_3 \hbox{   and   } \omega = \widehat{\omega}.
\end{equation*}
In Section \ref{S4} we have also proved that the eigenfunctions of $\mathcal{L}_s$ are 
\begin{eqnarray}
\omega^1(x_1,x_3) &=& \sum_{-\infty}^\infty a_{t,sn+r}\cos(tx_1+ (sn+r)x_3)\cr
\omega^2(x_1,x_3) &=& \sum_{-\infty}^\infty a_{t,sn+r}\sin(tx_1+ (sn+r)x_3).
\end{eqnarray}
Hence,
$$\omega^1(x_1,x_3) + i\omega^2(x_1,x_3) = e^{itx_1}\sum_{n=-\infty}^\infty a_{t,sn+r}e^{i(sn+r)x_3}.$$
We can find an unstable vector valued eigenfunction of the operator $L_2(\vec{v}_0)$ in the form \eqref{eigenfunction} by applying the operator $\curl_n \Delta_x^{-1}$ to the above equation and get that
$$\omega(x_1,x_3) = (\omega_1(x_3),\omega_3(x_3))^T e^{itx_1}.$$

For the $3$-D instability analysis we need to repeat the construction of an unstable eigenmode on the torus $\widehat{\mathbb{T}}^2_a = [0,2\pi/|\widehat{a}|]\times[0,2\pi]$. For this purpose we apply Theorem \ref{THLower2D} on $\widehat{\mathbb{T}}^2_a$ to obtain that 
\begin{proposition}\label{Prop}
Let $r$ and $t':=t|\widehat{a}|$ belong to region $A(\delta)$:
\begin{equation}\label{Cond1}
t'^2+r^2<s^2/3, \, t'^2+(-s+r)^2>s^2, \, t'^2+(s+r)^2>s^2, \, t'\geqslant \delta s.
\end{equation}
Taking $\widetilde{f}_s$ and $\vec{v}_0$ in two dimension context as
$$\widetilde{f}_s(x_3) = (-\frac{\widehat{a}}{a\sqrt{2}\pi}\nu^2\lambda s^2\sin s x_3, \, 0)^T,\, \, \, \vec{v}(x_3) = (\frac{1}{\sqrt{2}\pi}\nu\lambda \sin s x_3,\, 0)^T.$$
Then there exists an unstable solution
\begin{equation}\label{Omega}
\omega(x_1,x_3) = (\omega_1(x_3),\omega_3(x_3))^T e^{it\varepsilon x_1} \hbox{  where  } x\in \widehat{\mathbb{T}}^2_a
\end{equation}
under the form \eqref{eigenfunction} of the operator \eqref{twoDoperator} on $\widehat{\mathbb{T}}^2_a$.
\end{proposition}
\begin{proof}
The proof is a consequence of Theorem \ref{THLower2D} by substituting $t':=|\widehat{a}|t$.
\end{proof}
It is convenient to single out a small rectangle $D$ in the $(t',r)$-plane inside the region given by \eqref{Cond1}:
\begin{equation}
|r|\leq c_2s, \, 0<c_3s\leq t' \leq c_4 s.
\end{equation}
Here $\delta = \delta^* \in (0,1/\sqrt{3})$ is fixed, and all the constants $c_i$ are absolute constants,
whose explicit values can be specified.
\begin{theorem}
We consider the linearized equation \eqref{Lineareq} on the $3$-torus $\mathbb{T}^3 = [0,2\pi]^3$ with right-hand side $f_s$ and stationary solution $\vec{v}_0$ given by \eqref{flow1} and \eqref{stasol}, where
\begin{equation}\label{lambda}
\lambda = \lambda_3(s) = \sqrt{2}\lambda_2(s) = \sqrt{2}c_1 s(1+\alpha^2s^2)^2.
\end{equation}
(where $\lambda_2(s)$ is given in Theorem \ref{THLower2D}). Then for each triple of integers $a,\, b,\, r$ satisfying
\begin{equation}\label{ConditionABR}
c_3 s\leq \widehat{a} =\sqrt{a^2+b^2} \leq c_4s, \, |r|\leq c_2 s,\,\, |b|\leq a,
\end{equation}
there exists an unstable solution of the linearized operator \eqref{SO}. The number of integers $(a,b,r)$ satisfied \eqref{ConditionABR} is of order $c_5s^3$, where
$$c_5 = \frac{1}{4}\pi c_2(c_4^2-c_3^2).$$
\end{theorem}
\begin{proof}
The proof is a slightly modification of \cite[Theorem 5.5]{Il21} for replacing $\gamma$ by $\Delta$. We fix $a,b$ and $r$ satisfy \eqref{ConditionABR}. Since the first two inequalities in \eqref{ConditionABR}, we have the pair $(t',r)\in D \subset A(\delta)$, where $t'=\widehat{a}.1$ (therefore, we set here $t = 1$).
Applying Squire's transformation we obtain a $2$-D linearized problem on the torus $\widehat{\mathbb{T}}^2_a$ of the form \eqref{twoDoperator} with $\widehat{\Delta} = \dfrac{\widehat{a}}{a}\Delta$. Using the third inequality in \eqref{ConditionABR} we have
$$\lambda = \sqrt{2}\lambda_2(s,\Delta) = \sqrt{2}\dfrac{a}{\widehat{a}} \dfrac{\widehat{a}}{a}\lambda_2(s)\geq \lambda_2(s,\widehat{\Delta}).$$
Since Proposition \ref{Prop}, we have that the $2$-D linearized problem \eqref{twoDoperator} has
an unstable eigenvalue. By using Lemma \ref{Lem} this deduces that the $3$-D linearized
problem \eqref{SO} has also unstable eigenvalue on the standard torus $\mathbb{T}^3=[0,2\pi]^3$.
Our proof is completed.
\end{proof}
Now we give the lower bound of the attractor's dimensions of the modified Leray-alpha equation \eqref{ModCH} on the $3$-D torus $\mathbb{T}^3=[0,2\pi]^3$ in the following theorem.
\begin{theorem}
Let the right-hand side in \eqref{ModCH} be \eqref{flow1}. The dimension of the corresponding attractor $\mathcal{A}=\mathcal{A}_s$ of \eqref{ModCH} satisfies the lower bound
$$\mathrm{dim}_F\mathcal{A} \geq c_6 \frac{G^{\gamma}}{\alpha^{3(1-\gamma)}},$$
where $G= |f|/\nu^2$ is Grashof number and $0<\alpha \ll 1$, $0\ll \gamma<1$.
\end{theorem}
\begin{proof}
We consider only the case $0<\alpha \ll 1$. Since $s$ is at our disposal we take $s=1/\alpha$. Therefore, we obtain for $\lambda$ in \eqref{lambda}, hence $f_s$ that
$$\lambda = c_6 \frac{1}{\alpha}, \,\, \, \norm{f_s}^2_{L^2} = \frac{\nu^4}{\alpha^6}.$$
Finally, we have
$$\mathrm{dim}_F\mathcal{A} \geq c_6s^3 = c_6\frac{1}{\alpha^3}.$$
Putting $G= |f|/\nu^2 = \alpha^{-3}$, we establish that
$$\mathrm{dim}_F\mathcal{A} \geq c_6\frac{G^{\gamma}}{\alpha^{3(1-\gamma)}} \,\,\, (0\ll \gamma <1).$$
\end{proof}
\begin{remark}
By combining with the upper bound of the attractor's dimension for $3$-D modified Leray-alpha equation obtained in \cite[Theorem 6]{IlLuTi}:
$$\mathrm{dim}_F\mathcal{A} \leq c_7 \left( \frac{G'}{\alpha} \right)^{3/2},$$
where $G'= G/\lambda_1^{3/4}\simeq G$. We obtain the two-side estimate of the attractor's dimension
$$c_5\frac{G^{\gamma}}{\alpha^{3(1-\gamma)}} \leq \mathrm{dim}_F\mathcal{A} \leq c_8 \left( \frac{G}{\alpha} \right)^{3/2}\,\,\, (0 \ll \gamma <1).$$
Therefore, the sharp upper bound of $\mathrm{dim}_F {\mathcal{A}}$ must equivalent to $G^\kappa$ with the power $1<\kappa <\dfrac{3}{2}$.
\end{remark}

\end{document}